\documentclass[11pt]{amsart}
\usepackage{graphicx, color}
\usepackage{amscd}
\usepackage{amsmath}
\usepackage{amsfonts}
\usepackage{amssymb}
\usepackage{mathrsfs}
\newtheorem{theorem}{Theorem}

\newtheorem{lemma}[theorem]{Lemma}
\newtheorem{proposition}[theorem]{Proposition}
\newtheorem{remark}[theorem]{Remark}

\numberwithin{equation}{section}

\renewcommand{\leq}{\leqslant}

\renewcommand{\geq}{\geqslant}

\begin{document}
\pretolerance=10000


\title[Nehari method for locally Lipschitz functionals]{Nehari method for locally Lipschitz functionals with examples in problems in the space of bounded variation functions}

\thanks{M.T.O.P. is supported by Fapesp 2014/16136-1, 2015/12476-5, 2017/01756-2 and CNPq 442520/2014-0. G.M.F. is supported by CNPq and FAPESP}


\author[G.M. Figueiredo]{Giovany M. Figueiredo}
\address{Departamento de Matem\'atica, Universidade de Bras\'ilia,
70910-900, Bras\'ilia, Brazil}
\email{\tt giovany@unb.br}

\author[M.T.O. Pimenta]{Marcos T.  O. Pimenta}
\address{Departamento de Matem\'atica e Computa\c{c}\~ao,
          Faculdade de Ci\^encias e Tecnologia, Universidade Estadual Paulista - UNESP,
          19060-900 Presidente Prudente - SP, Brazil}
\email{\tt pimenta@fct.unesp.br}



\begin{abstract}
In this work we prove some abstract results about the existence of a minimizer for locally Lipschitz functionals, without any assumption of homogeneity, over a set which has its definition inspired in the Nehari manifold. As applications we present a result of existence of ground state bounded variation solutions of problems involving the 1-Laplacian and the mean-curvature operator, where the nonlinearity satisfies mild assumptions.

\end{abstract}

\maketitle

\tableofcontents

\section{Introduction and some abstract results}

\hspace{.5cm} 

Since its appearance, the Nehari method has been used in a number of situations in order to get ground state solutions to elliptic problems. In \cite{SzulkinWeth} Szulkin and Weth have elegantly described in a systematic way the real essence of the Nehari method. Their main theorem, describing in an abstract framework sufficient conditions to get ground state solutions, has been applied to study problems like
$$
\left\{
\begin{array}{rl}
\displaystyle - \Delta_p u - \lambda|u|^{p-2}u & = f(x,u) \quad \mbox{in
$\Omega$,}\\
u & = 0 \quad \mbox{on
$\partial\Omega$,}\\
\end{array} \right.
$$
where $\lambda < \lambda_1$, $\lambda_1$ the first eigenvalue of $-\Delta_p$ and $f$ a subcritical power-type nonlinearity. Also, they considered elliptic problems in $\mathbb{R}^N$ like
$$
\left\{
\begin{array}{ll}
\displaystyle - \Delta u + V(x)u & = f(x,u) \quad \mbox{in
$\mathbb{R}^N$,}\\
u(x) \to 0, & |x| \to \infty,
\end{array} \right.
$$
under some conditions on $V$ and $f$.

These two problems (and other two also considered in \cite{SzulkinWeth}) have a common feature which allows the use of the standard Nehari method described by Szulkin and Weth. In fact, the energy functionals associated to them has the "principal part" $p-$homogenous. This condition, by its side, has been dropped out in the paper of Figueiredo and Ramos \cite{FigueiredoQuoirin}, in which they conclude the same as Szulkin and Weth, but with weaker assumptions. The applications of Figueiredo and Ramos results treat with non homogenous elliptic problems. They consider, for instance, quasilinear elliptic problems which includes the $p-q$ Laplacian operator, Kirchhoff-type problems and an anisotropic equation (see \cite{FigueiredoQuoirin} for details).

Leaving aside the questions about the homogeneity, a common feature which seems to never be dropped out in results like these, is the differentiability of $\Phi$. In fact this is a natural assumption that researchers have always considered when dealing with sets which resemble a Nehari manifold, since some directional derivatives of the functional is involved in its definition.

In this work we present some abstract results whose assumptions are enough to give a reasonable sense to a Nehari set, which contains all the critical points of functionals which are neither $C^1$, nor even differentiable and where all minimizers are critical points. More specifically, we deal with functionals defined in Banach spaces, which are written like a combination of a locally Lipschitz and a $C^1$ functional and satisfy some natural assumptions. In fact, despite the lack of smoothness in our functional, we ask it to have at least directional derivatives in some directions which are enough to get the result. Although this can sound a little bit artificial, in a lot of examples this assumption is satisfied. For example, as we will show later in Section \ref{applications}, in problems involving the $1-$Laplacian (the formal limit of $\Delta_p$ as $p \to 1$), the mean curvature operator and some other operators derived from these ones, the associated energy functionals are not differentiable in all directions.

 Our main results are the following.

\begin{theorem}
\label{theorem1}
Let $E,F$ be Banach spaces, $E$ not necessarily reflexive, $E$ compactly embedded into $F$ and such that for all bounded sequence $(u_n) \subset E$ such that $u_n \to u$ in $F$, it holds that $u \in E$. Let $\Phi, I_0: E \to \mathbb{R}$, $I: F \to \mathbb{R}$ be functionals such that $\Phi = I_0 - I |_E$, where $I_0$ is locally Lipschitz continuous, $I \in C^1(E) \cap C^0(F)$, $I(0) = 0$ and for all $u \in E$, there exists the following limit
$$I_0'(su)u := \lim_{t \to 0} \frac{I_0(su + tu) - I_0(su)}{t}, \quad \forall s \in \mathbb{R}.$$
Suppose also the following conditions are satisfied:
\begin{itemize}
\item [$i)$] $I_0$ is lower semicontinuous in the topology of $F$ in the following sense, if $(u_n)$ is a bounded sequence in $E$, and $u \in E$ are such that, up to a subsequence, $u_n \to u$ in $F$, then $\displaystyle I_0(u) \leq \liminf_{n \to \infty} I_0(u_n)$;
\item [$ii)$] There exist $\rho, \alpha_0 > 0$, such that $\Phi(u) \geq \alpha_0 > \Phi(0)$, for every $u \in E$ with $\|u\| = \rho$;
\item [$iii)$] $\forall u \in E$, $\displaystyle \Phi(u) \geq \|u\| - I(u)$;
\item [$iv)$] $\displaystyle t \mapsto I_0'(tu)u$ is non-decreasing in $(0,+\infty)$ and $\displaystyle t \mapsto I'(tu)u$ is increasing in $(0,+\infty)$;
\item [$v)$] For each bounded sequence $(v_n) \subset E$ such that $v_n \to v \neq 0$ in $F$, it follows that 
$$\displaystyle \lim_{t \to \infty} \frac{\Phi(tv_n)}{t} = -\infty, \quad \mbox{uniformly for $n \in \mathbb{N}$.}$$
\end{itemize}
Then, the infimum of $\Phi$ on the following set
$$\mathcal{N} := \{u \in E\backslash\{0\}; \, \, I_0'(u)u = I'(u)u\}$$
is achieved.
\end{theorem}

It is important to say that, since Theorem \ref{theorem1} is going to be applied in non-reflexive Banach spaces, it is not possible to work with weak convergence in $E$, this way, being necessary to introduce the space $F$.

Since we are dealing with locally Lipschitz functionals, before saying something about the existence of critical points, we have to give the precise definition of it. We say that $u_0 \in E$ is a critical point of $\Phi$, if $0 \in \partial \Phi(u_0)$, where $ \partial \Phi(u_0)$ denotes the generalized gradient of $\Phi$ in $u_0$, as defined in \cite{Chang} for instance.

Finally, our last theorem ensures that all minimizers of $\Phi$ on the Nehari set, $\mathcal{N}$, are in fact critical points.
\begin{theorem}
\label{theorem3}
Suppose all conditions of Theorem \ref{theorem1} hold, then if $u_0 \in \mathcal{N}$ is such that $\displaystyle \Phi(u_0) = \min_{\mathcal{N}}\Phi$, then $u_0$ is a critical point of $\Phi$ in $E$.
\end{theorem}

As applications of these results we address the question of finding critical points of functionals involving the absolute value of the total variation of a function in the space of functions of bounded variation, $BV(\Omega)$, in a setting in which coerciveness and smoothness are lost. In fact, a lot of attention has been paid recently to the space $BV(\Omega)$, since it is the natural environment in which minimizers of many problems can be found, especially in problems involving interesting physical situations, like existence of minimal surfaces and in capillarity theory. In fact, this space turns to be the natural domain of relaxed versions of some functionals defined in usual Sobolev spaces. 

In \cite{DegiovanniMagrone}, Degiovanni and Magrone study the version of Br\'ezis-Nirenberg problem to the 1-Laplacian operator, corresponding to 
\begin{equation}
\left\{
\begin{array}{rl}
\displaystyle - \Delta_1 u & = \displaystyle  \lambda \frac{u}{|u|} + |u|^{1^*-2}u \quad \mbox{in
$\Omega$,}\\
u & = 0 \quad \mbox{on
$\partial\Omega$,}\\
\end{array} \right.
\label{Pintro1}
\end{equation}
where $1^* = N/(N-1)$ and $\displaystyle \Delta_1 u = \mbox{div}\left(\frac{\nabla u}{|\nabla u|}\right)$.  In this work, the authors extend $\Phi$ in a suitable $L^p(\Omega)$ space, which have better compactness properties, even tough the continuity of it is lost.

This kind of argument, which consists in extend the functional defined in $BV(\Omega)$, to some $L^p(\Omega)$ in order to recover the Palais-Smale condition, is generally used in dealing with $\Delta_1$ operator. For example, in \cite{Chang2}, Chang uses this approach to study the spectrum of the $1-$Laplacian operator, proving the existence of a sequence of eigenvalues.

In the first application we study the energy functional whose Euler-Lagrange equation is a relaxed form of 

\begin{equation}
\left\{
\begin{array}{rl}
\displaystyle - \Delta_1 u & = f(u) \quad \mbox{in
$\Omega$,}\\
u & = 0 \quad \mbox{on
$\partial\Omega$,}\\
\end{array} \right.
\label{P1}
\end{equation}
where $\Omega \subset \mathbb{R}^N$ is a smooth bounded domain, $N \geq 2$ and $f$ satisfies the following set of assumptions.
\begin{itemize}
\item [$(f_1)$] $f \in C^0(\mathbb{R})$;
\item [$(f_2)$] $\displaystyle f(s) = o(1)$, as $|s| \to 0$;
\item [$(f_3)$] there exist constants $c_1, c_2 > 0$ and $p \in (1,1^*)$ such that
$$|f(s)| \leq c_1 + c_2|s|^{p-1}, \quad s \in \mathbb{R};$$
\item [$(f_4)$] $\displaystyle \lim_{t \to \pm\infty}\frac{F(t)}{t} = \pm\infty$;
\item [$(f_5)$] $f$ is increasing for $s \in \mathbb{R}$.
\end{itemize}

Let us consider $I_0, I: BV(\Omega) \to \mathbb{R}$ defined by
$$I_0(u) = \int_\Omega |Du| + \int_{\partial\Omega} |u|d\mathcal{H}_{N-1}$$
and
$$I(u) = \int_\Omega F(u) dx.$$
It follows that $\Phi: BV(\Omega) \to \mathbb{R}$ given by $\Phi(u) = I_0(u) - I(u)$ is the functional whose Euler-Lagrange equation is a relaxed form of (\ref{P1}). Since $\Phi$ is just a locally Lipschitz functional (since $I_0$ lacks smoothness), we have first to explain what we mean by saying that $u \in BV(\Omega)$ is a critical point of $\Phi$. Since $I_0$ is a convex locally Lipchitz functional and $I \in C^1(BV(\Omega))$, then $u \in BV(\Omega)$ is going to be called a {\it critical point of $\Phi$}, or a {\it bounded variation solution of (\ref{P1})}, if $I'(u) \in \partial I_0(u)$, where $\partial I_0(u)$ denotes the subdifferential of the convex function $I_0$. This is equivalent to 
\begin{equation}
I_0(v) - I_0(u) \geq I'(u)(v-u), \quad \forall v \in BV(\Omega).
\label{BVsolution}
\end{equation}

In the second application of our main results, we study a prescribed mean-curvature problem given by
\begin{equation}
\left\{
\begin{array}{rl}
\displaystyle -\mbox{div}\left(\frac{\nabla u}{\sqrt{1+|\nabla u|^2}}\right) & =  \lambda f(u) \quad \mbox{in
$\Omega$,}\\
u & = 0 \quad \mbox{on
$\partial\Omega$,}\\
\end{array} \right.
\label{Pcurvaturamediaintro}
\end{equation}
where $\lambda > 0$, $\Omega \subset \mathbb{R}^N$ is a smooth bounded domain, $N \geq 2$ and $f$ satisfies the same set of assumptions $(f_1)-(f_5)$. By working in $BV(\Omega)$ itself, we get ground state solutions of (\ref{Pcurvaturamediaintro}) by proving that the energy functional $\Phi$ satisfies all the assumptions of Theorems \ref{theorem1} and \ref{theorem3}, where $\Phi$ is given by
$$\Phi(u) = \tilde{I}_0(u) - \lambda I(u),$$
where $\tilde{I}_0, I: BV(\Omega) \to \mathbb{R}$ are given by
$$\tilde{I}_0(u) = \int_\Omega\sqrt{1 + |Du|^2} + \int_{\partial\Omega} |u|d\mathcal{H}_{N-1}$$
and
$$I(u) = \int_\Omega F(u) dx.$$

 In \cite{Omari}, Obersnel and Omari performed a deep study of (\ref{Pcurvaturamediaintro}) with a nonlinearity involving a parameter $\lambda > 0$, considering several different situations, like nonlinearities subquadratic at $0$, sublinear at $+\infty$, subquadratic at $0$ and sublinear at $+\infty$, among others. The authors have applied different methods in each situation, like a truncation in the operator by a function $a$, which allows them to work in $H^1_0(\Omega)$.

In \cite{Marzocchi}, Marzocchi has considered (\ref{Pcurvaturamediaintro}) which slightly stronger hypothesis and proved, by using an extension of the functional of $BV(\Omega)$ to some space $L^p(\Omega)$ and proving an abstract result of Ambrosetti-Rabinowitz type (see the celebrated paper \cite{Rabinowitz}), the existence of a sequence of solution whose energy diverges to $+\infty$. Note that any information about a ground state solution has been given in \cite{Marzocchi}, in such a way that our result can be seen in this sense as an improvement of both \cite{Omari} and \cite{Marzocchi}.

In \cite{Szulkin} the author study all the usual minimax theorems like Mountain Pass Theorem, Saddle Point Theorem, etc., for functionals written like $\Phi = I_0 - I$, where $I_0$ is locally Lipschitz, lower semicontinuous and $I$ is $C^1$. However, any mention of arguments which resembles a Nehari set has been made. Moreover, these results in general have among their assumptions, the Palais-Smale condition, which is known to be very hard to be proved for functionals defined in $BV(\Omega)$.

To finish, the wide range of situations in which Nehari method has been applied in problems involving operators like $-\Delta$ and $-\Delta_p$, together with the lack of study of different geometric situations when dealing with the 1-Laplacian and mean-curvature operators, lead us to think that this paper can give the tools to shed some light in several questions which could be raised to problems in $BV(\Omega)$, in analogy to situations involving $-\Delta$ and $-\Delta_p$ operators.

In Section 2 we prove the abstract results stated in this introduction. In Section 3, we start with a short subsection, in order to furnish the basic notation and results about functionals defined in $BV(\Omega)$. Later, we present two applications of our abstract results. To finish, in an Appendix we present an alternative proof of Theorem \ref{theorem3} which uses a version of the Lagrange Multiplier rule to locally Lipschitz functionals (see \cite{Clarke}), in the concrete case where the function $f$ is $C^1$, rather than just $C^0$.

\section{Proof of the abstract results}

\begin{proof}[Proof of Theorem \ref{theorem1}]
Let us start proving that $\mathcal{N} \neq \emptyset$. In order to do this let us prove that for all $w \in E \backslash \{0\}$, there exists a unique $t_w > 0$ such that $t_w w \in \mathcal{N}$.

For $w \in E \backslash \{0\}$ consider $\gamma_w : \mathbb{R}_+ \to \mathbb{R}$ defined by 
$$\gamma_w(t) = \Phi(t w).$$
Note that $\gamma_w$ is a smooth function such that, $\gamma_w(0) = \Phi(0)$ and $\gamma_w'(t) = I_0'(tw)w - I'(tw)w$. By $ii)$, for $t_0 = \rho/\|w\|$, it follows that
$\gamma_w(t_0) \geq \alpha_0 > \Phi(0)$. Moreover, by $v)$, it follows that
$$\lim_{t \to \infty} \gamma_w(t) = -\infty.$$
Then there exists a $t_w > 0$ such that $\gamma_w'(t_w) = 0$ and consequently, $t_w w \in \mathcal{N}$.

In order to verify the uniqueness of such a $t_w$, note that, supposing that $\gamma_w'(t) = \gamma_w'(s) = 0$, for $t,s > 0$, then it follows that
$$I_0'(tw)w - I'(tw)w = 0$$
and
$$I_0'(sw)w - I'(sw)w = 0.$$
Then
$$I_0'(tw)w - I_0'(sw)w = I'(sw)w - I'(tw)w$$
and by $iv)$ it follows that $t = s$.

Note that for all $u \in \mathcal{N}$, it follows by $ii)$
\begin{equation}
\Phi(u) = \max_{t \geq 0}\Phi(t u) \geq \Phi\left(\frac{\rho}{\|u\|}u\right) \geq \alpha_0 > \Phi(0).
\label{Eqnova}
\end{equation}
Then there exists $(u_n) \subset \mathcal{N}$ such that
$$
\lim_{n \to \infty}\Phi(u_n)  = \inf_{\mathcal{N}}\Phi =: c.
$$

Note that there exists $\delta > 0$ such that 
\begin{equation}
\|u\| \geq \delta \quad \mbox{for all $u \in \mathcal{N}$.}
\label{Neharibaixo}
\end{equation}
In fact, otherwise it would exists $(w_n) \subset \mathcal{N}$ such that $w_n \to 0$ in $E$, which would contradict (\ref{Eqnova}).

Let us prove that the minimizing sequence $(u_n)$ is bounded in $E$. Assume by contradiction that $\|u_n\| \to \infty$, as $n \to \infty$ and let $\displaystyle v_n = \frac{u_n}{\|u_n\|}$. Since $(v_n) \subset E$ is bounded, the compactness of the embedding $E \hookrightarrow F$, implies that there exists $v \in F$ such that $v_n \to v$ in $F$, up to a subsequence. Then it follows that $I(v_n) \to I(v)$.

If $v = 0$, then by $iii)$, for all $t \geq 0$
\begin{eqnarray*}
c + o_n(1) & = & \Phi(u_n) = \Phi(v_n\|u_n\|)\\
& = & \max_{s \geq 0}\Phi(s v_n)\\
& \geq & \Phi(t v_n)\\
& \geq & t - I(tv_n)\\
& = & t + o_n(1),
\end{eqnarray*}
which is a clear contradiction with the fact that $c \in \mathbb{R}$.

If $v \neq 0$, then by $v)$
$$
o_n(1) = \lim_{n \to \infty} \frac{c}{\|u_n\|} = \lim_{n \to \infty}\frac{\Phi(u_n)}{\|u_n\|} = \lim_{n \to \infty}\frac{\Phi(v_n\|u_n\|)}{\|u_n\|} = -\infty.
$$
Then we get a contradiction.

Hence $(u_n)$ is bounded in $E$ and then there exists $u \in F$ such that $u_n \to u$ in $F$, up to a subsequence. Then, by hypothesis, it follows that $u \in E$.

If $u = 0$, then by (\ref{Neharibaixo}), for all $t \geq 0$,
\begin{eqnarray*}
c + o_n(1) & = & \Phi(u_n)\\\
& = & \max_{s \geq 0}\Phi(s u_n)\\
& \geq & \Phi(t u_n)\\
& \geq & t\|u_n\| - I(tu_n)\\
& \geq & t \delta - I(tu_n)\\
& = & t\delta + o_n(1),
\end{eqnarray*}
which give us a contradiction.

Hence $u \neq 0$ and let us consider $t_u u \in \mathcal{N}$. Then, by $i)$
\begin{eqnarray*}
c & \leq & \Phi(t_u u)\\
& = & I_0(t_u u) - I(t_u u)\\
& \leq & \liminf_{n \to \infty} I_0(t_u u_n) - \lim_{n \to \infty}I(t_u u_n)\\
& = & \liminf_{n \to \infty} \Phi(t_u u_n)\\
& \leq & \liminf_{n \to \infty} \Phi(u_n) = c
\end{eqnarray*}
and then $\Phi(t_u u) = c$ and $t_u u \in \mathcal{N}$, which proves the theorem.
\end{proof}

Now let us present the proof of our second result.

\begin{proof}[Proof of Theorem \ref{theorem3}]
Suppose by contradiction that $0 \not \in \partial \Phi(u_0)$, then $\beta(u_0) > 0$, where $\beta(u_0) = \inf\{\|z\|_*; \, z \in \partial \Phi(u_0)\}$. Since $u \mapsto \beta(u)$ is lower semicontinuous (see \cite{Chang} for instance), it follows that there exists $\kappa > 0$ such that
\begin{equation}
\beta(u) > \frac{\beta(u_0)}{2} > 0, \quad \forall u \in B_\kappa(u_0).
\label{betaestimate}
\end{equation}
Let us denote $\displaystyle J = \left[1 - \frac{\kappa}{4},1 + \frac{\kappa}{4}\right] \subset \mathbb{R}$ and define $g: J \to E$ by
$$g(t) = tu_0$$
Denoting $\displaystyle c := \Phi(u_0) = \inf_{\mathcal{N}}\Phi$, note that
$$
\Phi(g(t)) < c, \quad \forall t \neq 1.
$$
Moreover, note that
$$
\max\left\{\Phi\left(g\left(1 - \frac{\kappa}{4}\right)\right),\Phi\left(g\left(1 + \frac{\kappa}{4}\right)\right)\right\} = c_0 < c.
$$

By using the version of the Deformation Lemma to locally Lipschitz functionals, without the Palais-Smale condition (see \cite{Dai}), there exists $\epsilon > 0$ such that
$$
\epsilon < \min\left\{\frac{c-c_0}{2},\frac{\kappa \beta(u_0)}{16}\right\},
$$
and an homeomorphism $\eta: E \to E$ such that
\begin{itemize}
\item [$i)$] $\eta(x) = x$ for all $x \not \in  \Phi^{-1}([c-\epsilon_0,c+\epsilon_0]) \cap B_\kappa(u_0)$;
\item [$ii)$] $\eta(\Phi_{c+\epsilon}\cap B_{\kappa/2}(u_0)) \subset \Phi_{c-\epsilon}$;
\item [$iii)$] $\Phi(\eta(x)) \leq \Phi(x)$, for all $x \in E$.
\end{itemize}

Let us define now $h: J \to E$ by $h(t) = \eta(g(t))$ and two functions, $\Psi_0, \Psi_1 : J \to \mathbb{R}$ by
$$\Psi_0(t) = \Phi'(tu_0)u_0$$
and
$$\Psi_1(t) = \frac{1}{t}\Phi'(h(t))h(t).$$

Since for $\displaystyle t \in \left\{\left(1 - \frac{\kappa}{4}\right),\left(1 + \frac{\kappa}{4}\right)\right\}$, $\Phi(g(t)) \leq c_0 < c-\epsilon_0$, then $h(t) = \eta(g(t)) = g(t) = tu_0$ for $t \in \displaystyle \left\{\left(1 - \frac{\kappa}{4}\right),\left(1 + \frac{\kappa}{4}\right)\right\}$. Hence
\begin{equation}
\Psi_0(t) = \Psi_1(t), \quad \forall t \in \left\{\left(1 - \frac{\kappa}{4}\right),\left(1 + \frac{\kappa}{4}\right)\right\}.
\label{fronteira}
\end{equation}

By degree theory, $d(\Psi_0,J,0) = 1$ and, taking into account (\ref{fronteira}), we have that $d(\Psi_1,J,0) = 1$. Then we have that there exists $t \in J$ such that $h(t) \in \mathcal{N}$.
This implies that 
$$c \leq \Phi(h(t)) = \Phi(\eta(g(t))).$$
But note that $\Phi(g(t)) < c + \epsilon$ and also $g(J) \subset B_{\kappa/2}(u_0)$. Then, by $ii)$
$$ \Phi(\eta(g(t))) < c - \epsilon$$
which contradicts the last inequality. Then the result follows.
\end{proof}

\section{Applications}
\label{applications}
\subsection{Preliminaries}
\label{Preliminares}
First of all let us introduce the space of functions of bounded variation, $BV(\Omega)$. We say that $u \in BV(\Omega)$, or is a function of bounded variation, if $u \in L^1(\Omega)$, and its distributional derivative $Du$ is a vectorial Radon measure, i.e., 
$$BV(\Omega) = \left\{u \in L^1(\Omega); \, Du \in \mathcal{M}(\Omega,\mathbb{R}^N)\right\}.$$
It can be proved that $u$ belongs to $BV(\Omega)$ is equivalent to
$$\int_\Omega |Du| := \sup\left\{\int_\Omega u \mbox{div}\phi dx; \, \, \phi \in C^1_c(\Omega,\mathbb{R}^N), \, \mbox{s.t.} \, \, \|\phi\|_\infty \leq 1\right\} < +\infty.$$

The space $BV(\Omega)$ is a Banach space when endowed with the norm
$$\|u\|_{BV(\Omega)} := \int_\Omega |Du| + \|u\|_1,$$
which is continuously embedded into $L^r(\Omega)$ for all $r \in [1,N/(N-1)]$ and is compactly embedded for $r \in [1,N/(N-1))$ (see \cite{Buttazzo}[Theorems 10.1.3, 10.1.4]).

Because of Trace Theorem \cite{Buttazzo}[Theorem 10.2.1] and the continuous embedding of $BV(\Omega)$ into $L^1(\Omega)$, it follows that 
$$\|u\| := \int_\Omega |Du| + \int_{\partial \Omega}|u|d\mathcal{H}_{N-1},$$
where $\mathcal{H}_{N-1}$ denotes the usual $(N-1)-$dimensional Hausdorff measure, defines in $BV(\Omega)$ an equivalent norm, which we use as the standard norm in $BV(\Omega)$ from now on.

It can be proved that $I_0, \tilde{I}_0:BV(\Omega) \to \mathbb{R}$ given by
\begin{equation}
I_0(u) = \int_\Omega |Du| + \int_{\partial\Omega} |u|d\mathcal{H}_{N-1}
\label{I_0}
\end{equation}
and
\begin{equation}
\tilde{I}_0(u) = \int_\Omega \sqrt{1 + |Du|^2} + \int_{\partial\Omega} |u|d\mathcal{H}_{N-1}
\label{I_0tilde}
\end{equation}
are convex functionals and Lipschitz continuous in $BV(\Omega)$, where 
$$\int_\Omega \sqrt{1 + |Du|^2}:= sup\left\{\int_\Omega \left(g_{N+1} +  \sum_{i=1}^{N}u\frac{\partial g_i}{\partial x_i} \right)dx; \, \, g_i \in C^1_c(\Omega), \,  \|(g_1,...,g_{N+1})\|_\infty \leq 1\right\}.$$

Although it follows from a classical argument, for the sake of completeness, let us prove in the following result that $\tilde{I}_0$ is lower semicontinuous with respect to the $L^r(\Omega)$ topology, for $r \in [1,N/(N-1)]$.

\begin{lemma}
$\tilde{I}_0$ is lower semicontinuous with respect to the $L^r(\Omega)$ topology, for $r \in [1,N/(N-1)]$, i.e., if $(u_n) \subset BV(\Omega)$, $u \in BV(\Omega)$ are such that $u_n \to u$ in $L^r(\Omega)$, for $r \in [1,N/(N-1)]$, then
$$\tilde{I}_0(u) \leq \liminf_{n \to \infty}\tilde{I}_0(u_n).$$
\label{I_0tildesemicontinuo}
\end{lemma}
\begin{proof}
In fact, the difficulty here is the presence of the integral on $\partial \Omega$ in $\tilde{I}_0$, since in the absence of this term, this is classical result which can be found in \cite{Giusti}[Theorem 14.2] for example.

Let $B \subset \mathbb{R}^N$ be a ball such that $\overline{\Omega} \subset B$. Given $u \in BV(\Omega)$, define
$$\overline{u}(x)= \left\{
\begin{array}{rl}
u(x), & x \in \Omega\\
0, & x \in B\backslash \overline{\Omega},
\end{array} \right.
$$
and note that $\overline{u} \in BV(B)$ and $D\overline{u} = Du|_\Omega + u \nu \mathcal{H}_{N-1}|_{\partial \Omega}$ (see \cite{Buttazzo}[Example 10.2.2]), where $\nu$ denotes the inner normal vector to $\partial \Omega$.

Then it follows that
\begin{equation}
\int_B\sqrt{1 + |D\overline{u}|^2} = \int_\Omega\sqrt{1 + |Du|^2} + \int_{\partial \Omega}|u|d\mathcal{H}_{N-1} + |B\backslash \Omega|.
\label{lowersemicontinuous}
\end{equation}

Now let $(u_n)$ and $u$ be like in the statement. Then the corresponding extensions $(\overline{u}_n) \subset BV(B)$, $\overline{u} \in BV(B)$ and $\overline{u}_n \to \overline{u}$ in $L^r(B)$. Moreover, for each $g \in C^1_c(B, \mathbb{R}^{N+ 1})$, we have that
\begin{eqnarray*}
\int_B \left(g_{N+1} +  \sum_{i=1}^{N}\overline{u}\frac{\partial g_i}{\partial x_i} \right)dx & = & \lim_{n \to \infty} \int_B \left(g_{N+1} +  \sum_{i=1}^{N}\overline{u}_n\frac{\partial g_i}{\partial x_i}\right)dx \\
& \leq & \liminf_{n \to \infty} \int_B\sqrt{1 + |D\overline{u}_n|^2}.
\end{eqnarray*}
Then
$$
\int_B\sqrt{1 + |D\overline{u}|^2} - |B\backslash \overline{\Omega}| \leq \liminf_{n \to \infty} \int_B\sqrt{1 + |D\overline{u}_n|^2} - |B\backslash \overline{\Omega}|
$$
and by (\ref{lowersemicontinuous}) it follows that
$$
\tilde{I}_0(u) \leq \liminf_{n \to \infty} \tilde{I}_0(u_n).
$$
\end{proof}
\begin{remark}
Using the same arguments it is possible to prove the same statement for $I_0$.
\label{lowersemicontinuousI_0}
\end{remark}

It follows also that $BV(\Omega)$ is a {\it lattice}, i.e., if $u,v \in BV(\Omega)$, then $\max\{u,v\}, \min\{u,v\} \in BV(\Omega)$ and also
\begin{equation}
I_0(\max\{u,v\}) + I_0(\min\{u,v\}) \leq I_0(u) + I_0(v), \quad \forall u,v \in BV(\Omega)
\label{lattice}
\end{equation}
and
\begin{equation}
\tilde{I}_0(\max\{u,v\}) + \tilde{I}_0(\min\{u,v\}) \leq \tilde{I}_0(u) + \tilde{I}_0(v), \quad \forall u,v \in BV(\Omega).
\label{lattice}
\end{equation}

For a vectorial Radon measure $\mu \in \mathcal{M}(\Omega,\mathbb{R}^N)$, we denote by $\mu = \mu^a + \mu^s$ the usual decomposition stated in the Radon Nikodyn Theorem, where $\mu^a$ and $\mu^s$ are, respectively, the absolute continuous and the singular parts with respect to the $N-$dimensional Lebesgue measure $\mathcal{L}^N$. We denote by $|\mu|$, the absolute value of $\mu$, the scalar Radon measure defined like in \cite{Buttazzo}[pg. 125]. By $\displaystyle \frac{\mu}{|\mu|}(x)$ we denote the usual Lebesgue derivative of $\mu$ with respect to $|\mu|$, given by
$$\frac{\mu}{|\mu|}(x) = \lim_{r \to 0}\frac{\mu(B_r(x))}{|\mu|(B_r(x))}.$$ 
 
In order to remark the proprieties of differentiability (or the lack of it) that $I_0$ and $\tilde{I}_0$ have, let us recall the result of Anzellotti in \cite{Anzellotti}. For $g: \Omega \times \mathbb{R}^N \to \mathbb{R}$, let us define
$$g^0(x,p) = \lim_{t \to 0^+}g\left(x,\frac{p}{t}\right)t.$$
Suppose that $g$ is differentiable in $p$ for all $x \in \Omega$, $p \in \mathbb{R}^N$ and $g^0(x,p)$ is differentiable for all $x \in \Omega$, $p \in \mathbb{R}^N\backslash \{0\}$ and also that there exists $M > 0$ such that
$$|g_p(x,p)| \leq M, \quad |g^0_p(x,p)| \leq M.$$
Then $\mathcal{J}_g: BV(\Omega) \to \mathbb{R}$ defined by 
$$\mathcal{J}_g(u) = \int_\Omega g(x,Du) := \int_\Omega g(x,(Du)^a(x))dx + \int_\Omega g^0\left(x,\frac{Du}{|Du|}(x)\right)|Du|^s$$ 
is differentiable at the point $u \in BV(\Omega)$ in the direction $v \in BV(\Omega)$ if and only if $|Dv|^s$ is absolutely continuous with respect to $|Du|^s$, and in such a case one has
$$\mathcal{J}_g'(u)v = \int_\Omega g_p(x,(Du)^a(x))(Dv)^a(x)dx + \int_\Omega g^0_p\left(x,\frac{Du}{|Du|}(x)\right)\frac{Dv}{|Dv|}(x)|Dv|^s.$$ 
 
Since 
$$I_0(u) = \mathcal{J}_g(u) + \int_{\partial \Omega}|u|d\mathcal{H}_{N-1}$$
with $g(x,p) = g^0(x,p) = |p|$ and 
$$\tilde{I}_0(u) = \mathcal{J}_g(u) + \int_{\partial \Omega}|u|d\mathcal{H}_{N-1}$$
with $g(x,p) = \sqrt{1 + |p|^2}$ and $g^0(x,p) = |p|$, we have that, given $u \in BV(\Omega)$,
$I_0'(u)v$ is well defined for every $v \in BV(\Omega)$ such that $|Dv|^s$ is absolutely continuous with respect to $|Du|^s$ and $v(x) = 0$, $\mathcal{H}_ {N-1}-$ a.e. on the set $\{x \in \partial\Omega; \, u(x) = 0\}$ and we have that
\begin{equation}
I_0'(u)v = \int_\Omega \frac{(Du)^a(Dv)^a}{|(Du)^a|}dx + \int_\Omega \frac{Du}{|Du|}(x)\frac{Dv}{|Dv|}(x)|(Dv)|^s + \int_{\partial \Omega}\mbox{sgn}(u) v d\mathcal{H}_{N-1}.
\label{derivadaI_0}
\end{equation}

Analogously, we have that, given $u \in BV(\Omega)$,
$\tilde{I}_0'(u)v$ is well defined for every $v \in BV(\Omega)$ such that $|Dv|^s$ is absolutely continuous with respect to $|Du|^s$ and $v(x) = 0$, $\mathcal{H}_ {N-1}-$ a.e. on the set $\{x \in \partial\Omega; \, u(x) = 0\}$ and we have that
\begin{equation}
\tilde{I}_0'(u)v = \int_\Omega \frac{(Du)^a(Dv)^a}{\sqrt{1 + |(Du)^a|^2}}dx + \int_\Omega \frac{Du}{|Du|}(x)\frac{Dv}{|Dv|}(x)|(Dv)|^s + \int_{\partial \Omega}\mbox{sgn}(u) v d\mathcal{H}_{N-1}.
\label{derivadatildeI_0}
\end{equation}

Now let us just make precise the sense of solutions that we consider in this work. Regarding (\ref{Pcurvaturamediaintro}), the energy functional associated to it is $\Phi: BV(\Omega) \to \mathbb{R}$ given by
$$\Phi(u) = \tilde{I}_0(u) - \lambda I(u),$$
where 
\begin{equation}
I(u) = \int_\Omega F(u)dx.
\label{F}
\end{equation}

Since $I \in C^1(BV(\Omega))$ and $\tilde{I}_0$ is Lipschitz continuous, we say that $u_0 \in BV(\Omega)$ is a solution of (\ref{Pcurvaturamediaintro}) if $0 \in \partial \Phi(u_0)$, where $\partial \Phi(u_0)$ denotes the generalized gradient of $\Phi$ in $u_0$, as defined in \cite{Chang}. It follows that this is equivalent to $\lambda I'(u_0) \in \partial \tilde{I}_0(u_0)$ and, since $\tilde{I}_0$ is convex, this is written as
\begin{equation}
\tilde{I}_0(v) - \tilde{I}_0(u_0) \geq \lambda I'(u_0)(v-u_0), \quad \forall v \in BV(\Omega).
\label{BVsolutionMC}
\end{equation}
Hence all $u_0 \in BV(\Omega)$ such that (\ref{BVsolutionMC}) holds is going to be called a bounded variation solution of (\ref{Pcurvaturamediaintro}). In this work we are not going to obtain weak solutions in $W^{1,1}_0(\Omega)$ of  (\ref{Pcurvaturamediaintro}), since no regularity result is going to be proved here, but if we could do so, it would be possible to prove that all bounded variation solutions which belong to $W^{1,1}_0(\Omega)$, in fact satisfies the weak form of (\ref{Pcurvaturamediaintro}), namely
$$\int_\Omega \frac{\nabla u_0 \nabla v}{\sqrt{1 + |\nabla u_0|^2}}dx = \lambda \int_\Omega f(u_0)v dx, \quad \forall v \in W^{1,1}_0(\Omega).$$

Analogously, we say that $u_0 \in BV(\Omega)$ is a bounded variation solution of (\ref{P1}), if
$$
I_0(v) - I_0(u_0) \geq I'(u_0)(v-u_0), \quad \forall v \in BV(\Omega).
$$

\subsection{A problem involving $\Delta_1$}
\label{Section1Laplaciano}

Let us consider the problem of finding critical points of the functional $\Phi: BV(\Omega) \to \mathbb{R}$ given by
$$\Phi(u) = \int_{\Omega} |Du| + \int_{\partial\Omega} |u|d\mathcal{H}_{N-1} - \int_{\Omega} F(u) dx,$$
where $\Omega \subset \mathbb{R}^N$ is a smooth bounded domain, $N \geq 2$, $\displaystyle F(t) := \int_0^t f(s)ds$, and $f$ satisfy the following assumptions
\begin{itemize}
\item [$(f_1)$] $f \in C^0(\mathbb{R})$;
\item [$(f_2)$] $\displaystyle f(s) = o(1)$, as $|s| \to 0$;
\item [$(f_3)$] there exist constants $c_1, c_2 > 0$ and $p \in (1,1^*)$ such that
$$|f(s)| \leq c_1 + c_2|s|^{p-1}, \quad s \in \mathbb{R};$$
\item [$(f_4)$] $\displaystyle \lim_{t \to \pm\infty}\frac{F(t)}{t} = \pm\infty$;
\item [$(f_5)$] $f$ is increasing for $s \in \mathbb{R}$.
\end{itemize}

In fact the critical points of $\Phi$ are weak solutions of a relaxed form of
\begin{equation}
\left\{
\begin{array}{rl}
\displaystyle - \Delta_1 u & = f(u) \quad \mbox{in
$\Omega$,}\\
u & = 0 \quad \mbox{on
$\partial\Omega$,}\\
\end{array} \right.
\label{P1laplaciano}
\end{equation}
where $\displaystyle \Delta_1 u = \mbox{div}\left(\frac{\nabla u}{|\nabla u|}\right)$.

Let us consider $I_0, I: BV(\Omega) \to \mathbb{R}$ functionals given by (\ref{I_0}) and (\ref{F}), respectively and define $\Phi: BV(\Omega) \to \mathbb{R}$ by
$$\Phi(u) = I_0(u) - I(u).$$

Clearly the operator $\Delta_1$ is highly singular and some words about this imprecise way to define it have to be stated. The first step is to extend the functionals $I_0$, $I$ and $\Phi$ to $L^{1^*}(\Omega)$, defining $\overline{I_0}, \overline{I}, \overline{\Phi}: L^{1^*}(\Omega) \to \mathbb{R}$, where
$$
\overline{I_0}(u) = 
\left\{
\begin{array}{ll}
I_0(u), & \mbox{if $u \in BV(\Omega)$},\\
+\infty, & \mbox{if $u \in L^{1^*}(\Omega)\backslash BV(\Omega)$},
\end{array}
\right.
$$
$$
\overline{I}(u) = \int_{\Omega}F(u)dx
$$
and $\overline{\Phi} = \overline{I_0} - \overline{I}$. It is easy to see that $\overline{I}$ belongs to $C^1(L^{1^*}(\Omega))$ and that $\overline{I_0}$ is convex and lower semicontinuous in $L^{1^*}(\Omega)$. Hence the subdifferential of $\overline{I_0}$ is well defined. The following is a crucial result in obtaining an Euler-Lagrange equation satisfied by the critical points of $\Phi$.

\begin{lemma}
If  $u \in BV(\Omega)$ is such that $0 \in \partial \Phi(u)$, then $0 \in \partial \overline{\Phi}(u)$.
\end{lemma}
\begin{proof}
Suppose that $u \in BV(\Omega)$ is such that $0 \in \partial \Phi(u)$. Then $u$ satisfies (\ref{BVsolution}).
Let us verify that
$$
\overline{I_0}(v) - \overline{I_0}(u) \geq \overline{I}'(u)(v-u), \quad \forall v \in L^{1^*}(\Omega).
$$
For $v \in L^{1^*}(\Omega)$, note that:
\begin{itemize}
\item if $v \in BV(\Omega) \cap L^{1^*}(\Omega)$, then
\begin{eqnarray*}
\overline{I_0}(v) - \overline{I_0}(u) & = & I_0(v) - I_0(u)\\
& \geq & I'(u)(v-u)\\
& = & \int_{\Omega}f(u)(v-u)dx\\
& = & \overline{I}'(u)(v-u);
\end{eqnarray*}

\item if $u \in L^{1^*}(\Omega)\backslash BV(\Omega)$, since $\overline{I_0}(v) = +\infty$ and $\overline{I_0}(u) < +\infty$, it follows that
\begin{eqnarray*}
\overline{I_0}(v) - \overline{I_0}(u) & = & +\infty\\
& \geq & \overline{I}'(u)(v-u).
\end{eqnarray*}
\end{itemize}
Therefore the result follows.
\end{proof}

Let us assume that $u \in BV(\Omega)$ is a bounded variation solution of (\ref{P1laplaciano}). Since $0 \in \partial \Phi(u)$, by the last result it follows that $0 \in \partial \overline{\Phi}(u)$. Since $\overline{I_0}$ is convex and $\overline{I}$ is smooth, it follows that $\overline{I}'(u) \in \partial \overline{I_0}(u)$. By the very definition of the subdifferential, there exist $z^* \in L^N(\Omega)$ such that $z^* \in \partial \overline{I_0}(u)$ and
$$
\overline{I}'(u) = z^* \quad \mbox{ in $L^N(\Omega)$.}
$$
By \cite{Kawohl}[Proposition 4.23, pg. 529], it follows that there exist $z \in L^\infty(\Omega, \mathbb{R}^N)$ such that $\|z\|_\infty \leq 1$,
\begin{equation}
-\mbox{div}{z} = z^* \quad \mbox{ in $L^N(\Omega)$}
\label{eulerlagrange1}
\end{equation}
and 
\begin{equation}
 -\int_{\Omega}u \mbox{div}z dx = \int_{\Omega}|Du|,
 \label{eulerlagrange2}
 \end{equation}
 where the divergence in (\ref{eulerlagrange1}) has to be understood in the distributional sense. Therefore, it follows from (\ref{eulerlagrange1}) and (\ref{eulerlagrange2}) that $u$ satisfies
\begin{equation}
\left\{
\begin{array}{l}
\exists z \in L^\infty(\Omega,\mathbb{R}^N), \, \, \|z\|_\infty \leq 1,\, \,  \mbox{div}z \in L^N(\Omega), \\ \\ 
-\int_{\Omega}u \mbox{div}z dx = \int_{\Omega}|Du|,\\ \\
-\mbox{div} z = f(u), \quad \mbox{a.e. in $\Omega$}.
\end{array}
\right.
\label{eulerlagrangeequation}
\end{equation}
Hence, (\ref{eulerlagrangeequation}) is the precise version of (\ref{P1laplaciano}).

Note that $I_0$ is Lipschitz continuous in $BV(\Omega)$ and $I \in C^1(BV(\Omega))$. Moreover, $BV(\Omega)$ is compactly embedded into $L^p(\Omega)$, $p$ as in $(f_3)$, $I(0) = 0$ and, for all $u \in BV(\Omega)$ and $s \in \mathbb{R}$, by (\ref{derivadaI_0}) we have that
$$
\begin{array}{lll}
\displaystyle I_0'(su)u & = &\displaystyle \lim_{t \to 0} \frac{I_0(su + tu) - I_0(su)}{t}\\
& = & \displaystyle\int_\Omega \frac{(D(su))^a(Du)^a}{|(D(su))^a|}dx + \int_\Omega \frac{D(su)}{|D(su)|}(x)\frac{Du}{|Du|}(x)|Du|^s + \int_{\partial \Omega}\mbox{sgn}(su) u d\mathcal{H}_{N-1}\\
& = & \displaystyle \int_\Omega |(Du)^a|dx + \int_\Omega|(Du)^s| + \int_{\partial \Omega}|u|d\mathcal{H}_{N-1}
\end{array}
$$
and then $I_0'(u)u = I_0(u)$, for all $u \in BV(\Omega)$.

Let us define the Nehari set by
\begin{eqnarray*}
\mathcal{N} & = & \left\{u \in BV(\Omega)\backslash\{0\}; \, I_0'(u)u = I'(u)u \right\}\\
& = &  \left\{u \in BV(\Omega)\backslash\{0\}; \, \int_{\Omega} |Du| + \int_{\partial\Omega} |u|d\mathcal{H}_{N-1} = \int_\Omega f(u)udx \right\}.
\end{eqnarray*}

In the following result we prove that all nontrivial critical points of $\Phi$ belong to $\mathcal{N}$.

\begin{lemma}
If $u_0 \in BV(\Omega)$, $u_0 \neq 0$ and $0 \in \partial \Phi(u_0)$, then $u \in \mathcal{N}$.
\label{criticoNehari}
\end{lemma}
\begin{proof}
If $0 \in \partial \Phi(u_0)$, then
$$I_0(v) - I_0(u_0) \geq \int_\Omega f(u_0)(v-u_0)dx, \quad \forall v \in BV(\Omega).$$
For $t > 0$, by taking $v = u_0 + tu_0$ in the last expression and calculating the limit as $t \to 0^+$ we get
$$I_0(u_0) = \lim_{t \to 0^+}\frac{I_0(u_0 + tu_0) - I_0(u_0)}{t} \geq \int_\Omega f(u_0)u_0 dx.$$
Doing the same for $t < 0$ we get
$$I_0(u_0) = \lim_{t \to 0^-}\frac{I_0(u_0 + tu_0) - I_0(u_0)}{t} \leq \int_\Omega f(u_0)u_0 dx$$
from where it follows the equality in both expressions above. Hence $u_0 \in \mathcal{N}$.
\end{proof}

Note that by the last result, if we manage to prove that the infimum of $\Phi$ in $\mathcal{N}$ is achieved and it is a critical point, then we would get a nontrivial critical point of $\Phi$ with lowest energy among all nontrivial ones, then, it would be a ground state bounded variation solution of (\ref{P1laplaciano}). In order to do so, let us verify that $\Phi$ satisfies all the conditions of Theorem \ref{theorem1}.

Note that in Remark \ref{lowersemicontinuousI_0} we have proved that $I_0$ satisfies $i)$ of Theorem \ref{theorem1}.

For $ii)$, first of all note that $(f_2)$ and $(f_3)$ imply that for all $\epsilon > 0$, there exists $C_\epsilon$ such that
\begin{equation}
|F(s)| \leq \epsilon|s| + C_\epsilon|s|^p, \quad \mbox{for all $s \in \mathbb{R}$.}
\label{fepsilon}
\end{equation}
Then, the embeddings of $BV(\Omega)$ and (\ref{fepsilon}) imply that
\begin{eqnarray*}
\Phi(u) & \geq &(1 - C\epsilon)\|u\| - CC_\epsilon\|u\|^p\\
&  = & \|u\|(1 - C\epsilon - Cc_\epsilon \|u\|^{p-1})\\
& = & \rho(1 - C\epsilon - CC_\epsilon \rho^{p-1}) =: \alpha_0 > 0 = \Phi(0),
\end{eqnarray*}
where $\|u\| = \rho$ and $\epsilon$, $\rho$ are positive and small enough.

By definition of $\Phi$ it follows that $iii)$ holds, since in this case we have the equality being satisfied.

In order to verify $iv)$, just note that 
$$t \mapsto I'(tu)u = \int_{\Omega}f(tu)udx$$
is increasing in $(0, +\infty)$ by $(f_5)$. Also, by (\ref{derivadaI_0}), 
$$t \mapsto I_0'(tu)u = I_0(u)$$
is constant, then a non-decreasing function in $(0,+\infty)$.

Finally, to verify $v)$, let $(v_n) \subset BV(\Omega)$ and $v \in L^p(\Omega)\backslash\{0\}$ such that $v_n \to v$ in $L^p(\Omega)$. Since $\Phi(u) = \|u\| - \int_\Omega F(u)dx$, it is enough to prove that
$$\lim_{t \to \infty}\int_\Omega \frac{F(tv_n)}{t} = + \infty \quad \mbox{uniformly in $n \in \mathbb{N}$.}$$
Let $\Gamma = \{x \in \Omega; \, v(x) \neq 0\}$ and note that $|\Gamma| > 0$. Then by Fatou Lemma, it follows that, for all $t > 0$,
$$\int_\Gamma \frac{F(tv)}{t}dx \leq \liminf_{n \to \infty}\int_\Omega \frac{F(tv_n)}{t}dx.$$
Then, by $(f_4)$ we have that
$$\liminf_{t \to \infty} \liminf_{n \to \infty}\int_\Omega \frac{F(tv_n)}{t}dx \geq \liminf_{t \to \infty} \int_\Gamma \frac{F(tv)}{t}dx  = +\infty .$$
But this means that for every $M > 0$, there exist $t_0 > 0$ and $n_0 \in \mathbb{N}$ such that
$$\int_\Omega \frac{F(tv_n)}{t}dx \geq M, \quad \mbox{for every $t > t_0$ and $n \geq n_0$,}$$
which proves $v)$.

Now, since all conditions of Theorem \ref{theorem1} are satisfied, it follows the existence of $u_0 \in \mathcal{N}$ such that
$$\Phi(u_0) = \inf_{v \in \mathcal{N}} \Phi(v).$$
Moreover, by Theorem \ref{theorem3}, it follows that $u_0$ is a critical point to $\Phi$ and then, a ground state bounded variation solution of (\ref{P1laplaciano}).

\subsection{A problem involving the mean curvature operator}
\label{meancurvaturesection}
In this section we deal with the following prescribed mean-curvature problem

\begin{equation}
\left\{
\begin{array}{rl}
\displaystyle -\mbox{div}\left(\frac{\nabla u}{\sqrt{1+|\nabla u|^2}}\right) & =  \lambda f(u) \quad \mbox{in
$\Omega$,}\\
u & = 0 \quad \mbox{on
$\partial\Omega$,}\\
\end{array} \right.
\label{Pcurvaturamedia}
\end{equation}
where $\Omega \subset \mathbb{R}^N$ is a smooth bounded domain, $N \geq 2$ and $f: \mathbb{R} \to \mathbb{R}$ is assumed to satisfy the following set of assumptions
\begin{itemize}
\item [$(f_1)$] $f \in C^0(\mathbb{R})$;
\item [$(f_2)$] $\displaystyle f(s) = o(1)$, as $|s| \to 0$;
\item [$(f_3)$] there exists constants $c_1, c_2 > 0$ and $p \in (1,1^*)$ such that
$$|f(s)| \leq c_1 + c_2|s|^{p-1}, \quad s \in \mathbb{R};$$
\item [$(f_4)$] $\displaystyle \lim_{t \to \pm \infty}\frac{F(t)}{t} = \pm \infty$;
\item [$(f_5)$] $f$ is increasing for $s \in \mathbb{R}$.
\end{itemize}

Let us consider $\tilde{I}_0, I: BV(\Omega) \to \mathbb{R}$ given by (\ref{I_0tilde}) and (\ref{F}), respectively and let us define $\Phi: BV(\Omega) \to \mathbb{R}$ by 
$$\Phi(u) = \tilde{I}_0(u) - \lambda I(u).$$

Note that $\tilde{I}_0$ is Lipschitz continuous in $BV(\Omega)$ and $I \in C^1(BV(\Omega))$. Moreover, $BV(\Omega)$ is compactly embedded into $L^p(\Omega)$ and $I(0) = 0$. Note also that
\begin{equation}
\|u\| \leq \int_{\Omega}  \sqrt{1 + |Du|^2} + \int_{\partial\Omega} |u|d\mathcal{H}_{N-1} \leq \|u\| + |\Omega|.
\label{I_0estimate}
\end{equation}
Moreover, for all $u \in BV(\Omega)$ and $s \in \mathbb{R}$, by (\ref{derivadatildeI_0}) we have that
\begin{equation}
\label{derivadaI_0tilde}
\begin{array}{lll}
\displaystyle \tilde{I}_0'(su)u & = &\displaystyle \lim_{t \to 0} \frac{\tilde{I}_0(su + tu) - \tilde{I}_0(su)}{t}\\
& = & \displaystyle\int_\Omega \frac{(D(su))^a(Du)^a}{\sqrt{1 + |(D(su))^a|^2}}dx + \int_\Omega \frac{D(su)}{|D(su)|}(x)\frac{Du}{|Du|}(x)|Du|^s + \int_{\partial \Omega}\mbox{sgn}(su) u d\mathcal{H}_{N-1}\\
& = & \displaystyle \int_\Omega \frac{((Du)^a)^2}{\sqrt{\frac{1}{s^2} + |(Du)^a|^2}}dx + \int_\Omega|(Du)^s| + \int_{\partial \Omega}|u|d\mathcal{H}_{N-1}.
\end{array}
\end{equation}

Let us define the Nehari set by
\begin{eqnarray*}
\mathcal{N} & = & \left\{u \in BV(\Omega)\backslash\{0\}; \, \tilde{I}_0'(u)u = \lambda I'(u)u \right\}\\
& = & \left\{u \in BV(\Omega)\backslash\{0\}; \, \int_{\Omega} \frac{((Du)^a)^2}{\sqrt{1+ ((Du)^a)^2}} + \int_{\Omega}|Du|^s + \int_{\partial\Omega} |u|d\mathcal{H}_{N-1} =  \lambda \int_\Omega f(u)udx \right\}\\
\end{eqnarray*}

In the following result we state that all nontrivial critical points of $\Phi$ belong to $\mathcal{N}$. Its proof is totally analogous of Lemma \ref{criticoNehari}.

\begin{lemma}
If $u_0 \in BV(\Omega)$, $u_0 \neq 0$ and $0 \in \partial \Phi(u_0)$, then $u \in \mathcal{N}$.
\end{lemma}

As in the case of Section \ref{Section1Laplaciano}, if we manage to prove that the infimum of $\Phi$ in $\mathcal{N}$ is achieved and it is a critical point, then we would get a nontrivial critical point of $\Phi$ with lowest energy, then, it would be a ground state bounded variation solution of the prescribed mean-curvature problem (\ref{Pcurvaturamedia}). 

Note that in Lemma \ref{I_0tildesemicontinuo} we have proved that $\tilde{I}_0$ satisfies $i)$ in Theorem \ref{theorem1}.

For $ii)$, by Jensen's inequality, (\ref{fepsilon}), the embeddings of $BV(\Omega)$ and the Trace Theorem in $BV(\Omega)$, it follows that for all $u \in BV(\Omega)$ with $\|u\| = 1$,
\begin{eqnarray*}
\Phi(u) & \geq & \int_\Omega \sqrt{1 + ((Du)^a)^2} dx + \int_\Omega |Du|^s - \lambda C\epsilon \|u\| - \lambda CC_\epsilon\|u\|^p\\
& \geq & \sqrt{|\Omega|^2 + \left(\int_\Omega|(Du)^a|dx\right)^2}  + \int_\Omega |Du|^s - \lambda C\epsilon \|u\| - \lambda CC_\epsilon\|u\|^p\\
& \geq & \sqrt{|\Omega|^2 + \left(\int_\Omega|(Du)^a|dx\right)^2 + \left(\int_\Omega |Du|^s\right)^2} - \lambda C\epsilon \|u\| - \lambda CC_\epsilon\|u\|^p\\
& \geq & \sqrt{|\Omega|^2 + \frac{1}{4}\left(\int_\Omega|(Du)^a|dx + \int_\Omega |Du|^s\right)^2} - \lambda C\epsilon \|u\| - \lambda CC_\epsilon\|u\|^p\\
& \geq & \sqrt{|\Omega|^2 + C\|u\|^2} - \lambda C\epsilon \|u\| - \lambda CC_\epsilon\|u\|^p\\
& = & \sqrt{|\Omega|^2 + C} - \lambda(C\epsilon + CC_\epsilon)\\
& \geq & \alpha_0 > |\Omega| = \Phi(0),
\end{eqnarray*}
if $\lambda > 0$ is small enough. This proves $ii)$.

Note that (\ref{I_0estimate}) and the definition of $\Phi$, $\tilde{I}_0$ and $I$ imply that
$$\Phi(u) \geq \|u\| - \lambda I(u), \quad \forall u \in BV(\Omega),$$
proving that $iii)$ is satisfied.

For $iv)$, note that again, since $f$ in this section satisfies the same assumptions which those of Section \ref{Section1Laplaciano}, it follows that $ t \mapsto I'(tu)u$ is increasing in $(0,+\infty)$. For $\tilde{I}_0$, since
$$
\tilde{I}_0'(tu)u = \int_\Omega \frac{((Du)^a)^2}{\sqrt{\frac{1}{t^2} + |(Du)^a|^2}}dx + \int_\Omega|Du|^s + \int_{\partial \Omega}|u|d\mathcal{H}_{N-1},
$$
a simple analysis shows that $t \mapsto \tilde{I}_0'(tu)u$ is increasing and hence, a non-decreasing function in $(0,+\infty)$. Then $iv)$ in fact holds.

Finally, to verify $v)$, we just apply the same arguments employed in Section \ref{Section1Laplaciano}, together with (\ref{I_0estimate}).

Now, since all conditions of Theorem \ref{theorem1} are satisfied, it follows the existence of $u_0 \in \mathcal{N}$ such that
$$\Phi(u_0) = \inf_{v \in \mathcal{N}} \Phi(v).$$
Moreover, by Theorem \ref{theorem3}, it follows that $u_0$ is a critical point to $\Phi$ and then, a ground state bounded variation solution of (\ref{Pcurvaturamedia}).

\section{Appendix}

In the next result, we suppose that the nonlinearity satisfy $(f_1)-(f_5)$ and the additional hypothesis that $f \in C^1(\mathbb{R})$. In this situation, we can use a tool which has proven to be very effective in a number of situations when dealing with $-\Delta$ or $-\Delta_p$ operators and we think it has been underutilized in problems involving operators defined in the space of functions with bounded variation. In fact we are talking about the Lagrange multipliers rule for locally Lipschitz functionals like stated in \cite{Clarke}.

In fact the following proof just can be applied to the $1-$Laplacian operator, since in this case we have that $I_0'(u)u = I_0(u)$.

\begin{proposition}
Suppose that $f$ satisfy $(f_1)-(f_5)$ and moreover, that $f \in C^1(\mathbb{R})$. 
Let $I_0, I, \Phi: BV(\Omega) \to \mathbb{R}$, $I(u) = \int_\Omega F(u)dx$, $I_0$ like in Section \ref{Section1Laplaciano} and $\Phi = \bar{I}_0 - I$. Defining 
$$\mathcal{N} = \{u \in BV(\Omega)\backslash\{0\}; \, I'_0(u)u = I'(u)u\},$$
if $u_0 \in \mathcal{N}$ is such that
$$\Phi(u_0) = \inf_{v \in \mathcal{N}}\Phi(v),$$
then $0 \in \partial \Phi(u_0)$, i.e., $u_0$ is a critical point of $\Phi$ in $BV(\Omega)$.
\end{proposition}
\begin{proof}
First of all note that since $f \in C^1(\mathbb{R})$, then $I \in C^2(BV(\Omega))$. From \cite{Clarke}[Corollary 1], it follows that there exist $\Lambda, \mu \in \mathbb{R}$, not both equal to zero, such that
$$0 \in \mu \partial \Phi(u_0) + \Lambda \partial \mathcal{G}(u_0),$$\
where $\mathcal{G}: BV(\Omega) \to \mathbb{R}$ is a locally Lipschitz functional defined by 
$$
\mathcal{G}(u) = I_0'(u)u - I'(u)u = I_0(u) - I'(u)u.
$$

We claim that $\mu \neq 0$. In fact, on the contrary, if $\mu = 0$, then $\Lambda \neq 0$ and $0 \in \Lambda \partial \mathcal{G}(u_0)$. Then $0 \in \partial \mathcal{G}(u_0)$ and we have that
\begin{equation}
I_0(v) - I_0(u_0) \geq \int_\Omega\left(f'(u_0)u_0 + f(u_0)\right)(v - u_0)dx.
\label{eq1}
\end{equation}
For $t > 0$, by taking $v = u_0 + t u_0$ in (\ref{eq1}) and calculating the limit as $t \to 0^+$, we have
$$I_0'(u_0)u_0 = \lim_{t \to 0^+}\frac{I_0(u_0 + tu_0) - I_0(u_0)}{t} \geq \int_\Omega \left(f'(u_0)u_0^2 + f(u_0)u_0\right) dx.$$
Doing the same for $t < 0$ and $v = u_0 + t u_0$ in (\ref{eq1}), it follows that
$$I_0'(u_0)u_0 = \lim_{t \to 0^-}\frac{I_0(u_0 + tu_0) - I_0(u_0)}{t} \leq \int_\Omega \left(f'(u_0)u_0^2 + f(u_0)u_0\right) dx.$$
Then it follows that 
$$I_0'(u_0)u_0 = \int_\Omega \left(f'(u_0)u_0^2 + f(u_0)u_0\right)dx.$$
Since $I'(u_0)u_0 = \int_\Omega f(u_0)u_0dx$, it follows that
$$\int_\Omega f'(u_0)u_0^2 dx = 0,$$
which contradicts $(f_5)$. Hence in fact $\mu \neq 0$.

Since $\mu \neq 0$, we can suppose without lack of generality that $\mu = 1$ and then $0 \in \partial \Phi(u_0) + \Lambda \partial \mathcal{G}(u_0)$.

Now let us prove that $\Lambda  = 0$. 
Note that $0 \in \partial \Phi(u_0) + \Lambda \partial \mathcal{G}(u_0)$ is equivalent to the existence of $z_1^*, z_2^* \in (BV(\Omega))'$ such that $z_1^* \in \partial \Phi(u_0)$ and $z_2^* \in \Lambda\partial \mathcal{G}(u_0)$, such that $z_1^* + z_2^* = 0$ in $(BV(\Omega))'$. Then it follows that
\begin{equation}
I_0(v) - I_0(u_0) - \int_\Omega f(u_0)(v-u_0)dx \geq \langle z_1^*, v-u_0 \rangle, \quad \forall v \in BV(\Omega)
\label{eq2}
\end{equation}
and
\begin{equation}
\Lambda(I_0(v) - I_0(u_0)) - \Lambda\int_\Omega (f'(u_0)u_0 + f(u_0))(v - u_0)dx \geq \langle z_2^*, v - u_0 \rangle, \quad \forall v \in BV(\Omega).
\label{eq3}
\end{equation}

Now adding up (\ref{eq2}) and (\ref{eq3}) we get
\begin{equation}
(1+\Lambda)(I_0(v) - I_0(u_0)) \geq \int_\Omega ((1+\Lambda)f(u_0) + \Lambda f'(u_0)u_0)(v - u_0)dx, \quad \forall v \in BV(\Omega)
\label{eq4}
\end{equation}

For $t > 0$, by taking $v = u_0 + t u_0$ in (\ref{eq4}), dividing both sides by $t$ and calculating the limit as $t \to 0^+$, we have
$$(1+\Lambda) I_0(u_0) \geq \int_\Omega ((1+\Lambda)f(u_0) + \Lambda f'(u_0)u_0)u_0dx.$$
Doing the same for $t < 0$ we get
$$(1+\Lambda) I_0(u_0) \leq  \int_\Omega ((1+\Lambda)f(u_0) + \Lambda f'(u_0)u_0)u_0dx,$$
which implies that
$$(1+\Lambda) I_0(u_0) = \int_\Omega ((1+\Lambda)f(u_0) + \Lambda f'(u_0)u_0)u_0dx.$$
Since $u_0 \in \mathcal{N}$, we have that 
$$\Lambda\int_\Omega  f'(u_0)u_0^2 dx= 0,$$
which is possible, because of $(f_5)$, just if $\Lambda = 0$.
Then $0 \in \partial \Phi(u_0)$ and the results follows.
\end{proof}

\noindent \textbf{Acknowledgment.}\  

We would like to warmly thank Prof. Olimpio Hiroshi Miyagaki for several discussions about this subject.

This work was written while Giovany M. Figueiredo was as a Visiting Professor at FCT - Unesp in Presidente Prudente - SP. He would like to thanks the warm hospitality.

\end{document}